\documentclass[11pt]{article}

\usepackage{amssymb,amsthm,amsmath}


\usepackage[usenames,dvipsnames]{xcolor}
\usepackage[colorlinks,citecolor=blue,linkcolor=BrickRed]{hyperref}\usepackage[colorlinks,citecolor=blue,linkcolor=BrickRed]{hyperref}
\usepackage{thmtools,thm-restate}
\usepackage{comment}
\usepackage{cleveref}
\usepackage{fullpage}
\usepackage{pdfpages}
\usepackage{graphicx}
\graphicspath{ {./images/} }

\newtheorem{theorem}{Theorem}[section]
\newtheorem{corollary}[theorem]{Corollary}

\newtheorem{lemma}[theorem]{Lemma}
\newtheorem{assumption}[theorem]{Assumption}
\newtheorem{observation}[theorem]{Observation}
\newtheorem{definition}[theorem]{Definition}

\newtheorem{fact}[theorem]{Fact}
\newtheorem*{claim*}{Claim}

\newcommand{\poly}{\ensuremath{\mathsf{poly}}}
\newcommand{\E}{\mathbb{E}}

\newcommand{\R}{\mathbb{R}}

\newcommand{\polylog}{\mathsf{polylog}}
\newcommand{\Tr}{\mathsf{Tr}}

\newcommand{\diag}{\mathsf{diag}}

\global\long\def\norm#1{\left\Vert #1\right\Vert }

\newcommand{\disc}{\mathsf{disc}}
\newcommand{\p}{\mathbb{P}}

\newcommand{\ignore}[1]{{}}

\newcommand{\bad}{\mathsf{bad}}

\newcommand{\trunc}{\mathsf{trun}}

\newcommand{\fail}{\mathsf{fail}}
\newcommand{\safe}{\mathsf{safe}}
\newcommand{\dang}{\mathsf{dang}}

\allowdisplaybreaks

{\hspace*{\fill}$\Box$\par}

\title{An Improved Bound for the Beck-Fiala Conjecture\footnote{To appear in FOCS 2025. The result in this paper is subsumed by follow-up work \cite{BJ25b}.}}
\author{Nikhil Bansal\thanks{University of Michigan, Ann Arbor, MI, USA. \texttt{bansaln@umich.edu }.} 
 \and 
Haotian Jiang\thanks{University of Chicago, Chicago, IL, USA. \texttt{jhtdavid@uchicago.edu}.}}

\date{}

\begin{document}

\begin{titlepage}
\maketitle

\begin{abstract}
In 1981, Beck and Fiala \cite{BF81} conjectured that given a set system $A \in \{0,1\}^{m \times n}$ with degree at most $k$ (i.e., each column of $A$ has at most $k$ non-zeros), its combinatorial discrepancy $\mathsf{disc}(A) := \min_{x \in \{\pm 1\}^n} \|Ax\|_\infty$ is at most $O(\sqrt{k})$. Previously, the best-known bounds for this conjecture were either $O(k)$, first established by Beck and Fiala \cite{BF81}, or $O(\sqrt{k \log n})$, first proved by Banaszczyk \cite{Ban98}. 

\smallskip

We give an algorithmic proof of an improved bound of $O(\sqrt{k \log\log n})$ whenever $k \geq \log^5 n$, thus matching the Beck-Fiala conjecture up to  $O(\sqrt{\log \log n})$ for almost the full regime of $k$.

\end{abstract}

 \thispagestyle{empty}
\end{titlepage}

\thispagestyle{empty}
{\hypersetup{linkcolor=BrickRed}
\tableofcontents
}

\newpage
\setcounter{page}{1}

\section{Introduction}

Combinatorial discrepancy theory is the study of the following question: given a universe of elements $U=\{1,\ldots, n\}$ and a collection $\mathcal{S} = \{S_1, \ldots, S_m\}$ of subsets of $U$, how well can we partition $U$ into two pieces, so that all the sets in $\mathcal{S}$ are split as evenly as possible.
Formally, the combinatorial discrepancy of the set system $\mathcal{S}$ is defined as 
\begin{equation*}
    \mathrm{disc}(\mathcal{S}) := \min_{x: U \rightarrow \{-1,1\}^n} \max_{i \in [m]} \big|\sum_{j \in S_i}  x(j) \big| ,
\end{equation*}
where the partition $x: U \rightarrow \{-1,1\}^n$ is also called a {\em coloring}. 
Denoting by $A \in \{0,1\}^{m\times n}$ the incidence matrix of $\mathcal{S}$, i.e., $A_{ij} = 1$ if $j \in S_i$ and $0$ otherwise, we can write $\mathrm{disc}(\mathcal{S}) =  \mathrm{disc}(A) := \min_{x\in\{-1,1\}^n}\norm{{A}x}_{\infty}$. Discrepancy is a classical and well-studied topic with many applications in both mathematics and theoretical computer science \cite{Cha00,Mat09}.

\medskip
\noindent {\bf The Beck-Fiala Conjecture.}
A central and long-standing question has been to understand the discrepancy of bounded-degree set systems --- where each element appears in at most $k$ sets, i.e., where the matrix $A$ is $k$-column sparse. 
In their seminal work \cite{BF81}, Beck and Fiala showed that $\disc(A) \leq 2k-1$ for any such matrix $A$, and conjectured that $\disc(A) = O(\sqrt{k})$. This latter bound is the best possible up to a constant, in general.

\medskip
\noindent {\bf Non-Constructive Methods.} The study of the Beck-Fiala conjecture has led to various powerful and general techniques in discrepancy, that have led to numerous other applications. 

Beck developed the partial coloring method \cite{Bec81}, based on the pigeonhole principle and counting, which was later refined by Spencer \cite{Spe85} and Gluskin \cite{Glu89}. These methods show the existence of a good {\em partial} coloring, where a constant fraction of  elements are colored $\pm 1$,
with discrepancy $O(\sqrt{k})$.
Iterating this $O(\log n)$ times gives a full coloring, but loses an additional $\log n$ factor, i.e. the total discrepancy is $O(\sqrt{k} \log n)$. This method has been studied extensively over the last four decades, but it seems unclear how to use it to improve upon the $O(\sqrt{k} \log n)$ bound.

A different approach was developed by Banaszczyk \cite{Ban98}, based on deep ideas from convex geometry, that directly produces a full coloring with discrepancy $O(\sqrt{k \log n})$. 
This approach gives a beautiful connection between discrepancy and the Gaussian measure of an associated convex body. As shown by \cite{DGLN16}, it is equivalent to the existence of a distribution over colorings $x \in \{\pm 1\}^n$  such that the resulting discrepancy vector $Ax$ is $O(k)$-subgaussian.\footnote{This means that for each test vector $\theta \in \R^m$, the random discrepancy vector $Ax$ satisfies $\E[\exp(\langle \theta,Ax\rangle)]= \exp(O(k\|\theta\|_2^2))$.}
The extra $\sqrt{\log n}$ factor loss results from a union bound over all the $n$ rows, and seems inherent for techniques that only use subgaussianity.

\medskip
\noindent {\bf Algorithmic Methods.}
Interestingly, both these methods were originally non-constructive,
and did not give an efficient algorithm to actually find a good coloring. 
However, there has been a lot of recent progress in this direction, and several elegant algorithms are now known for both the partial coloring method 
\cite{Ban10,BS13,LM15,HSS14,Rot17,ES18,JSS23},
and also for Banaszczyk's approach \cite{DGLN16,BDG19,BDGL18,LRR17, BLV22, ALS21, PV23, HSSZ24}.\footnote{Some of the cited work here gave algorithmic proofs for both the partial coloring method and Banaszcyzk's approach, but we did not double-cite them.}
 These developments have led to many surprising applications and connections in computer science and mathematics. As this literature is already extensive, we discuss this very briefly in  \Cref{subsec:related_work}, and refer the reader to a recent survey \cite{Ban22} of these algorithmic aspects.

\medskip
\noindent {\bf Other Bounds.}
The bounds above incur a poly-logarithmic dependence on $n$. If no dependence on $n$ is allowed, the best bound is $2k - \log^*k$ due to Bukh \cite{Buk16}, which slightly improves upon the $2k-1$ bound of Beck and Fiala \cite{BF81} and subsequent refinements in \cite{BH97,Hel99}. 
A related result is the celebrated six-standard-deviations result of Spencer \cite{Spe85}, obtained independently by Gluskin \cite{Glu89}, which states that any set system with $m=O(n)$ sets has discrepancy $O(\sqrt{n})$. This can be viewed as a special case of the Beck-Fiala problem, where $k=m=O(n)$.\footnote{This is the only regime where the Beck-Fiala conjecture is currently known to be true.} 
For the general Beck-Fiala problem, however, not even a $(2 - \Omega(1))k$ bound is known to date. 

\subsection{Our Result and Approach}
We show the following improved bound for the Beck-Fiala problem, where the poly-logarithmic dependence on $n$ above is replaced by a $O(\sqrt{\log \log n})$ factor,\footnote{In the regime where $m \ll n$, the bound in \Cref{thm:beck-fiala-main} can be further improved to $O(\sqrt{k \log \log m})$ by a standard linear algebraic argument \cite{BF81,LSV86,Bar08}.} provided that $k$ is at least some $\polylog(n)$. More formally, we show the following result.

\begin{restatable}[Improved Bound for Beck-Fiala]{theorem}{BeckFialaMain}\label{thm:beck-fiala-main}
Let $A \in \{0, \pm 1\}^{m \times n}$ be a matrix where each column has at most $k$ non-zeros. If $k \geq \log^5 n$, then $\disc(A) = O(\sqrt{k \log \log n})$. Moreover, there is a polynomial time algorithm that finds a coloring $x \in \{\pm 1\}^n$ such that $\|A x\|_\infty = O(\sqrt{k \log \log n})$. 
\end{restatable}

This comes quite close to the conjectured $O(\sqrt{k})$ bound, for almost the full regime of $k$. 
As discussed above, nothing better than $O(\sqrt{k \log n})$ or $O(k)$ was known previously, even non-algorithmically, unless $k$ is very close to $n$.

\medskip\noindent \textbf{Our Approach in a Nutshell.} 
Our starting point is a recent algorithmic proof of Banaszczyk's $O(\sqrt{k \log n})$ bound \cite{Ban98}, due to Bansal, Laddha, and Vempala \cite{BLV22}, based on a stochastic process guided by a {\em barrier potential function}. 
Here, one defines a potential function that penalizes each row based on how close it is to violating the current discrepancy bound (aka the barrier). Starting from the coloring $x_0=0^n$, and a barrier at $b_0 > 0$, the process updates the coloring and the barrier in some clever way, so that
the total potential, i.e., sum of the potentials for each row, only decreases over time. The final discrepancy bound is given by $b_0$ plus the distance moved by the barrier.

There are various incarnations of this method, but they all suffer from a similar problem. As one only has control on the total potential,  a few bad rows may get {\em too close} to the barrier, and end up contributing to most of the total potential.
As the algorithm lacks more fine-grained control on which rows become bad, one needs to move the barrier quite rapidly over time to keep the total potential bounded, leading to a lossy discrepancy bound. 
We describe this more quantitatively in \Cref{sec:overview}.

Our key new insight is that one can  modify the algorithm so that, roughly speaking, the subset of rows that become bad as the process evolves, look somewhat {\em random}. In particular, even though we cannot avoid the potential being concentrated on a few bad rows, this subset of such rows is not too adversarial. This allows us to argue that the average size of such rows is relatively small, and combined with more refined barrier based arguments, we can obtain a substantially better discrepancy bound.

The condition $k\geq \polylog(n)$ in Theorem \ref{thm:beck-fiala-main} is needed as we need some concentration properties to control some parameters of the process over the entire $\poly(n)$ time steps of its evolution.

\subsection{Roadmap}
The rest of the paper is organized as follows.
We describe our algorithm and its analysis in \Cref{sec:proof_main}. The exposition there gets somewhat technical, 
which may perhaps obfuscate the main high level ideas. 
For ease of exposition and to explain the key ideas more clearly, in \Cref{sec:overview} we sketch a simpler algorithm with an almost self-contained proof (modulo some computations and technical lemmas) of a slightly weaker $O(\sqrt{ k} \log \log n)$ bound, when  $k\geq \log^5 n$.

\subsection{Further Related Works}
\label{subsec:related_work}

\noindent \textbf{Applications of Discrepancy Theory.}
The recent developments in algorithmic discrepancy theory mentioned earlier are based on a rich interplay between linear algebra, probability, convex geometry, and optimization, and have led to several surprising applications in areas such as differential privacy \cite{NTZ13,MN15}, combinatorics \cite{BG17,Nik17}, graph sparsification \cite{RR20}, approximation algorithms and rounding \cite{Rot16,BRS22,Ban24}, kernel density estimation \cite{PT20,CKW24}, randomized controlled trials \cite{HSSZ24}, and quasi-Monte Carlo methods \cite{BJ25}. 

\medskip
\noindent \textbf{Beyond the Worst-Case Setting.} The Beck-Fiala problem has also been studied extensively in random, pseudo-random, and smoothed settings \cite{EL19,HR19,FS20,BM20,Pot20,TMR20,BJM+22,ANW22} with the goal of improving upon the current worst case bounds.

\section{Overview of Ideas}
\label{sec:overview}

In this section, we describe our key ideas and give an essentially self-contained proof (with some computations and key technical lemmas deferred) of a slightly weaker $O(\sqrt{k} \log \log n)$ bound.     

We first describe the previous ideas, by giving a simple stochastic process, guided by a {\em barrier}-based potential, and proving a $O( \sqrt{k} \log^{3/2} n)$ discrepancy bound for it. Even though a better  $O(\sqrt{k} \log^{1/2} n)$ bound is known via this method \cite{BLV22},  this simple variant and its analysis will highlight the main bottleneck in previous approaches. To highlight our main new ideas, we then use the same potential, and show how to obtain an improved $O(\sqrt{k} \log \log n)$ bound.

\medskip 
\noindent {\bf Potential-Based Random Walk Framework.}
We first sketch the high-level framework.
Without loss of generality (see \Cref{assump:simplify_assump}),
let us assume that $m = n$, each column has exactly $k$ non-zeros, and the goal is to bound the one-sided discrepancy $\max_i \langle a_i, x \rangle$. 

 The algorithm maintains a fractional coloring $x_t \in [-1,1]^n$ at any time $t \in [0,n]$, starting with $x_0 = 0^n$. 
At any time $t$, define the {\em slack} of a row $a_i$ to be
\begin{equation}\label{eq:slack_intro}  
s_i(t)  := b_t -\langle a_i,x_t\rangle.
\end{equation}
Here $b_t > 0$ is the barrier and $\langle a_i,x_t\rangle$ is the discrepancy at time $t$.
 Initially, the barrier is at $b_0$. 
The algorithm will ensure that each slack $s_i(t)> 0$ at all times.

The coloring is updated over infinitesimal time steps $dt$ as $x_{t+dt} = x_t + v_t \sqrt{dt}$, where $v_t \perp x_t$ is a random unit vector with mean $\E[v_t]=0$, and supported on {\em alive} elements (i.e., uncolored elements $\{i \in [n]: |x_t(i)| < 1 \}$).
This ensures that $\|x_t\|_2^2=t$ at each time $t$, and the process ends at $t=n$.
We use $\mathcal{V}_t$ to denote the set of alive variables at time $t$, and let $n_t = |\mathcal{V}_t| $ be their number.

The barrier $b_t$ will be increased at some speed $c_tdt$ where $c_t \geq 0$ (chosen suitably), so the final barrier, which is also the final discrepancy bound (as slacks are non-negative), will be $b_n = b_0 + \int_{t=0}^n c_t dt$.

\subsection{An $O(\sqrt{k} \log^{3/2} n)$ Bound} 
We give a simple $O(\sqrt{k} \log^{3/2} n)$ bound using this framework. 
Consider the potential (or {\em weight}) 
\begin{equation}
\label{eq:overview-pot_i}
    \Phi_i(t) := \exp\big( b_0/s_i(t)\big)
\end{equation}
for row $i$, that
penalizes it based on how small its slack is relative to $b_0$, and blows up to infinity (rapidly) as $s_i(t)$ approaches $0$. 
Let $\Phi(t) := \sum_{i=1}^n \Phi_i(t)$ denote the total potential.
Initially, each $ \Phi_i(0)= e$ and $\Phi (0)=ne$.  

We will choose the unit coloring update vector $v_t$ so that $\Phi(t)$ only decreases over time, and hence $\Phi(t) \leq \Phi(0)$ at all times.\footnote{We will only ensure this in expectation and then prove a concentration bound, but let us ignore this technicality.} This will ensure that all the slacks $s_i(t)$ always stay positive (and the final discrepancy is at most $b_n)$.

\medskip
\noindent 
{\bf Evolution of Potential.} Let us see how the slacks and potential evolve when the coloring is updated at time $t$. The slack for row $i$ changes as $
ds_i(t) = c_t dt  - \langle a_i , v_t \rangle \sqrt{dt}$. 

Let $d\Phi_i(t):= \Phi_i(t+dt) - \Phi_i(t)$ denote the potential change for row $i$.
By a second order Taylor expansion (and ignoring $o(dt)$ terms)
and using that $\E[v_t]=0$, a standard calculation gives that
\begin{equation}
\label{eq:slack_overview_noenergy} \E[d\Phi_i(t)] \leq 
  \alpha_i(t) \Phi_i(t)  \big( - c_t + \alpha_i(t)  \E \langle a_i, v_t \rangle^2 \big ) d t  , 
\end{equation}
 where we denote $\alpha_i(t) := b_0/s_i(t)^2.$ 
 
Notice that $\E[\Phi_i(t)] \leq 0$ holds true whenever 
\begin{equation}
    \label{eq:overview:ct}
    c_t \geq \alpha_i(t) \E \langle a_i, v_t \rangle^2.
\end{equation}
This condition \eqref{eq:overview:ct} will play an important role.

\medskip
\noindent {\bf Choosing $c_t$.} 
We claim that choosing $c_t = O((k \log^2 n)/ b_0n_t)$ suffices to ensure that $\E[d\Phi_i(t)] \leq 0$, for all rows $i$ (and hence ensure that $\E[d\Phi(t)]\leq 0$).  
Indeed, observe the following:
\begin{enumerate}
    \item Suppose inductively that $\Phi (t) \leq \Phi(0)$. Then (trivially) each row potential  $\Phi_i(t) \leq \Phi(0) = O(n)$, and hence each slack 
 $s_i(t) = \Omega( b_0/\log n)$ by \eqref{eq:overview-pot_i}. This gives that $\alpha_i(t) = b_0/s_i(t)^2 =  O(\log^2 n)/b_0$.
\item  We can assume that each row $a_i$ has size $O(k)$ restricted to the alive elements $\mathcal{V}_t$ (by the standard trick of blocking large rows, say, of size $\geq 10k$). In particular, at any time there are at most $n_t/10$ such rows, and we can choose $v_t$ orthogonal to all of them. 
\item We can also assume that $v_t$ looks like a uniformly random unit vector supported on $\mathcal{V}_t$, 
using the technique of universal vector colorings \cite{BG17}, and satisfies that $\E \langle a_i, v_t \rangle^2   \approx \|a_i\|_2^2/n_t$. By the assumption above that $a_i$ has size $O(k)$, this quantity is $O(k/n_t)$.
\end{enumerate}
These observations together give that $\alpha_i(t) \E \langle a_i,v_t\rangle^2 \leq (\log^2 n)/b_0 \cdot (k/n_t)$ for each row $i$. By \eqref{eq:overview:ct}, we can set $c_t$ accordingly to obtain that\footnote{To ensure the integration converges, we integrate up to time $n-1$ instead, followed by a single rounding step.}
\begin{equation}
\label{eq:overview_bn} b_n = b_0 + \int_{t=0}^n c_t dt   = b_0 + \frac{O(k \log^2 n) }{b_0}\int_{t=0}^n \frac{1}{n_t}  dt  = b_0 + O\Big(\frac{k \log^3 n}{b_0}\Big) ,\end{equation} 
where we use that $n_t \geq  \lceil n-t \rceil $ for all $t$ since $\|x_t\|_2 = t$. Now, setting $b_0 = O(\sqrt{k} \log^{3/2} n)$ gives the claimed discrepancy bound.

\subsubsection{Where Can We Tighten This Analysis?} At first glance, the analysis above seems quite wasteful as: (i) we ensure that each $\E[\Phi_i(t)] \leq 0$ for each row $i$, even though we only need that $\E[\Phi(t)] \leq 0$, and 
 (ii) we use the bound $\Phi(t) \leq \Phi(0) = O(n)$ very weakly to infer that $\Phi_i(t) \leq \Phi(t)$ for each row $i$. 

However, as the following illustrative scenario shows, there is not much room to tighten the analysis.

\medskip
\noindent {\bf Bad Scenario.} Suppose at some time step $t$, we have the following state of the algorithm:

\,\,\,\,\, (i)   The potential $\Phi(t)$ is mostly distributed among a set of $2n_t$ rows (call these rows dangerous), 
 so that each $\Phi_i(t) \approx \Phi(0)/2n_t = \Theta (n/n_t)$ for them.

\,\,\,\,\, (ii) All these $2n_t$ rows have $\Omega(k)$ size.

For these dangerous rows, by \eqref{eq:overview-pot_i} we have  $s_i(t) \approx b_0/\log (n/n_t)$ and hence  $\alpha_i(t)\approx \log^2 (n/n_t)/b_0$ (recall that $\alpha_i(t) = b_0/s_i(t)^2$). Also, as the rows have size $\Omega(k)$, we have $\|a_i\|^2 = \Omega(k)$, and thus $\E[\langle a_i,v_t\rangle^2] = \Omega(k/n_t)$.

Now, as the algorithm can only make $v_t$ orthogonal to at most $n_t$ rows, the remaining $n_t$ dangerous rows will cause the overall potential rise, i.e.~$\E[d\Phi(t)] >0$, unless we set (by condition \eqref{eq:overview:ct}), 
\begin{align} \label{eq:c_t_overview_1}
c_t \gg \alpha_i(t)  \E[\langle a_i,v_t\rangle]^2 \approx (\log^2 (n/n_t)/b_0) \cdot (k/n_t) .
\end{align}
But this gives the same bound as in \eqref{eq:overview_bn}, as $\int_{t=0}^n \log^2 (n/n_t)/n_t dt  = \log^3 n$, giving the same $\Omega( \sqrt{ k \log^3 n})$ discrepancy bound.

To summarize, a problematic scenario is that (i) most of the potential $\Phi(t) \approx \Phi(0)$ ends up in just $2n_t$ rows, and (ii) each of these rows have size $\Omega(k)$. 
This is the bottleneck in all previous approaches, and it is  unclear how to rule it out, as the various rows and their potentials evolve in a very dependent way during the process.\footnote{The better bound of $O(\sqrt{k\log n})$ in \cite{BLV22} follows from a smarter way to define slack, with an additional {\em energy} term, and 
a smarter potential $\exp((b_0\log n)/s_i(t))$. But the $\sqrt{\log n}$ factor loss results from the same bad scenario.}

\subsection{Our Main Ideas}
\label{sec:main_idea_overview}

However, in an ideal scenario, one may optimistically hope that even though a small fraction of rows may become dangerous and have a large potential (e.g., $\Phi_i(t) \approx \Phi(0)/2n_t$ as above),
perhaps very few of the dangerous rows have size close to $k$.
To see why, notice that the potential $\Phi_i(t) \approx \Phi(0)/2n_t$ of a dangerous row is problematic only when $n_t \ll n$.
However, when $n_t \ll n$, the average row size is only
$kn_t/n \ll k$,\footnote{As $A$ has $ kn_t$ non-zero entries restricted to the alive columns $\mathcal{V}_t$, the average row size is $kn_t/n$.} which is much smaller than $k$.
Now, if we were lucky and most of dangerous rows had size much smaller than $k$,
then we could set $c_t$ to be much smaller than in \eqref{eq:c_t_overview_1} to obtain a better discrepancy bound.

Our key idea will be to enhance the algorithm in a way, so that this ideal scenario above essentially holds.
For example, one concrete instantiation of the property we will be able to maintain is --- 
\begin{align*}
 (*) \quad & \text{ With high probability our algorithm can ensure that, at each time $t$:}\\ 
 & \,\,\text{ At most $n_t/10$ rows have both (i) potential $\Phi_i(t) \gg \log^5 n$ and (ii) size $\gg k/\log^3n$.}
\end{align*}
We now describe how we achieve $(*)$ when $k \geq \log^5 n$. Then we will see how this gives an improved discrepancy bound.

\subsubsection{Almost Independent Evolution of Rows}
\label{sec:ind_evol_rows_overview}
 To achieve $(*)$, we further constrain the update vector $v_t$ at each time $t$ to also be orthogonal to the largest $n_t/10$ right singular vectors of the matrix $A$ restricted to $\mathcal{V}_t$.\footnote{Strictly speaking, we use a slightly modified version of the matrix $A$, see \Cref{sec:defn_process}.}
Roughly speaking, this substantially reduces the correlation among how the different rows evolve.

To elaborate a bit more, fix some time $t$. Consider some threshold $D \geq \polylog(n)$, and call a row dangerous if $\Phi_i(t) \geq D$. 
As $\Phi(t)=O(n)$, there can be at most $O(n/D)$ dangerous rows.
Now, if these $n/D$ dangerous rows were a  completely random subset of the $n$ rows, for any column $C_j$ of $A$, only about $ O(k/D)$ non-zero entries would lie in dangerous rows.\footnote{As each column $C_j$ of $A$ has $k$ non-zero entries initially.}

Surprisingly, we can show that our algorithm satisfies a similar guarantee --- with high probability, at each time $t$ and for each column $C_j$, the number of non-zero entries in $C_j$ that lie in the dangerous rows is 
 $O(k/D) + O(\log^2n)$. 
 When $k \geq \log^5 n$ and $D \geq \log^3 n$, this bound is $O(k / \log^3 n)$ as needed in $(*)$. 
 From this, the property $(*)$ above follows directly by a simple averaging argument.
 
 To prove the above guarantee, for each column $C_j$, we track the evolution of its ``weight"
 \[W_j(t) := \sum_{i:a_i(j)\neq 0} \min(D,\Phi_i(t)).\] Initially, clearly $W_j(0) = k \Phi (0)/n = O(k)$. In a bad case, if all $k$ non-zero entries in $C_j$ ended up lying in dangerous rows at time $t$, then we would have  $W_j(t) = k D$, which would be much larger than its initial weight $W_j(0)$.
 
The key observation, see \Cref{lem:conc_col_weight}, is that by blocking the large singular values, the weight $W_j(t)$ stays tightly concentrated around $O(k + D \log^2 n)$ for all $t$.
 As each dangerous row contributes $D$ to the weight, this directly gives the $O(k/D)+\log^2 n$ bound on the number of entries in dangerous rows per column.

\subsubsection{Improved Discrepancy Bound}
\label{subsec:improved_bound_overview}
We now sketch how property $(*)$ gives a $O(\sqrt{k} \log \log n)$ bound, using the potential \eqref{eq:overview-pot_i}, when $k \geq \log^5 n$.

As previously, it suffices to ensure that $\E[\Phi_i(t)]\leq 0$ for all rows $i$ and time $t$. 
Call a row $i$ {\em dangerous} at time $t$, if $\Phi_i(t) \geq \log^5 n$, and {\em safe} otherwise.

\medskip
\noindent{\bf Handling Dangerous Rows.}
We argue first that the dangerous rows are no longer problematic. As before, 
as $\Phi(t)\leq \Phi(0)$, each slack $s_i(t) = \Omega(b_0/\log n)$ and
any row has $\alpha_i(t)  = O((\log^2 n)/b_0)$. 
But now, by $(*)$, we have that $\E [\langle a_i,v_t\rangle^2] =O( \|a_i\|^2/n_t) = O(k/ n_t \log^3 n)$ for dangerous rows (the at most the $n_t/10$ rows with size larger than $\gg k/\log^3n$ are handled directly by choosing $v_t$ to be orthogonal to them, and we can pretend that they do not exist).  Thus setting $c_t := k/(b_0n_t \log n)$ already satisfies \eqref{eq:overview:ct} for such rows. 
For this $c_t$ we have 
\[ b_n = b_0 + \int_{t=0}^n c_t dt = b_0 + k/b_0,     \]
which, upon setting $b_0=\sqrt{k}$, in fact guarantees an $O(\sqrt{k})$ discrepancy for such rows!

\medskip
\noindent {\bf Handling Safe Rows.}
 The idea above does not help with safe rows, and we need an additional idea to handle them.
Let us consider this more closely.
By definition, $\Phi_i(t) \leq \log^5 n$ for such rows, and thus their slack $s_i(t) = \Omega(b_0/\log \log n)$ by \eqref{eq:overview-pot_i}, and hence $\alpha_i(t) = (\log \log n)^2/b_0$. Even though this $\alpha_i(t)$ is not too large, the problem is that now $\|a_i\|_2$ can be as large as $\Omega(k)$ for such rows. So we would need to set $c_t = \Omega(k (\log \log n)^2)/(b_0 n_t)$ in \eqref{eq:overview:ct} to ensure that $\E[\Phi_i(t)] \leq 0$. But this is too large, and plugging this $c_t$ into the integration $b_n = b_0 + \int c_t dt$ gives an $O(\sqrt{k \log n} \log \log n)$ discrepancy bound at best. 

Fortunately, we can use a different idea and add an {\em energy} term in the slack to help reduce the potential for safe rows. In particular, for a small row (of size $\leq 10k$) let us redefine the slack as
\[  s_i(t) := b_t - \langle a_i,x_t\rangle - \beta \sum_j a_i(j)^2 (1-x_t(j)^2).\]
We set $\beta := b_0/(20k)$, so that initially $s_i(0) \geq b_0/2$ (as each row has size at most $10k$) and thus $\Phi_i(0) \leq e^2$ for each row $i$.\footnote{For large rows, we just fix $s_i(t)=b_0/2$. Such rows do not affect anything as we choose $v_t$ orthogonal to them, and we can ignore these for the discussion here.}  

By a direct calculation (and ignoring some irrelevant terms) and as $\E \langle a_i,v_t\rangle^2 = O(\sum_j a_i(j)^2 v_t(j)^2)$ by the random-like property of $v_t$, one now has 
\begin{equation}
\label{eq:overview-dphi-2}
 \E[d \Phi_i(t)] \approx  \alpha_i(t) \Phi_i(t) \Big( -c_t + \big(\alpha_i(t)-\beta \big) \, \big(\sum_j a_i(j)^2 v_t(j)^2\big)\Big) dt.    
\end{equation}
The key difference here from \eqref{eq:slack_overview_noenergy} is the extra term $\beta \sum_j a_i(j)^2 v_t(j)^2 $ that helps decrease the potential.

\medskip
\noindent {\bf Finishing the Argument.} We are now done. 
For safe rows, as $\alpha_i(t) = O(\log\log n)^2/b_0$, we can simply set $\beta = \Omega((\log\log n)^2/b_0)$,\footnote{This requires that $\beta := b_0/(20k) = \Omega((\log\log n)^2/b_0)$, which is satisfied when $b_0 = \Omega(\sqrt{k} \log \log n)$.} which suffices to ensure that $\E [d\Phi_t(i) ]\leq 0$ in \eqref{eq:overview-dphi-2} (i.e.,~we do not use any help from 
$c_t$ for safe rows). 
For dangerous rows, as discussed earlier, setting $c_t = O(k/ b_0 n_t \log n)$ already suffices (even without any help from  the $\beta \sum_j a_i(j)^2 v_t(j)^2$ term).

Putting all this together, we can feasibly set $b_0 := \Theta( \sqrt{k} \log \log n)$ and $\beta := b_0/(20k)$ (this ensures that $\beta = \Omega((\log\log n)^2/b_0)$) and $c_t := O(k/ (b_0 n_t \log n))$, which gives the claimed discrepancy bound 
\[  b_n = b_0 + \int c_t dt = b_0 + k/b_0 = O(\sqrt{k} \log \log n).\]
This finishes the overview.
In the following \Cref{sec:proof_main}, we tighten up this analysis to give an improved $O(\sqrt{k \log\log n})$ discrepancy bound, under the same condition of $k\geq \log^5 n$, using a slightly better potential function and analysis.

\section{Proof of Our Result}
\label{sec:proof_main}

In this section, we prove our main result in \Cref{thm:beck-fiala-main} which is restated below. 

\BeckFialaMain*

Let $a_1, \cdots, a_m \in \{0, \pm 1\}^n$ be the rows of the matrix $A$. 
To simplify the presentation, 
we assume the following throughout, without loss of generality. 
\begin{assumption} \label{assump:simplify_assump}
\begin{enumerate}
    \item [(i)] It suffices to upper bound the one-sided discrepancy $\max_i \langle a_i, x \rangle$ (instead of $\max_i |\langle a_i,x\rangle|$). This can be done by appending the rows $-a_i$ to $A$ for all $i \in [m]$. 
    \item [(ii)] The number of columns and rows of $A$ are the same, i.e., $m = n$, and the number of non-zero entries in each column is exactly $k$.\footnote{This can be done trivially by adding a few dummy rows, and possibly some extra columns with non-zero entries only in these dummy rows.}
\end{enumerate} 
\end{assumption}
Notice that given any input matrix $A$ to \Cref{thm:beck-fiala-main}, the modifications in \Cref{assump:simplify_assump} to $A$
do not change the problem or the parameters by more than $O(1)$ factors. 
Throughout the paper, we use $i$ to index the rows of $A$ and $j$ to index the columns.

\subsection{The Algorithmic Framework}

As in previous works, our algorithm starts with the coloring $x_0 = 0^n$ and evolves it over time using (tiny) random increments, chosen suitably, until a final coloring in $\{-1,1\}^n$ is reached.

\medskip
\noindent {\bf Fractional Coloring.} Let  $x_t \in [-1,1]^n$ denote the
{\em fractional} coloring at time $t \in [0,n]$. Starting from $t=0$, the time will be updated in discrete increments of size $dt$.
It is useful to view $dt$ as infinitesimally small, though we will set $dt = 1/\poly(n)$ so that the algorithm runs in time $\poly(n)$.

Let $\mathcal{V}_t := \{j \in [n]: |x_t(j)| \leq 1- \frac{1}{2n}\}$ be the set of {\em alive} elements at time $t$, and $n_t := |\mathcal{V}_t|$ denote their number.
At time $t$, the coloring is updated by $d x_t = v_t \sqrt{dt}$ (so that $x_{t+dt}= x_t + dx_t$). We choose
$v_t \in \R^{\mathcal{V}_t}$, i.e., only the alive variables are updated. Notice that once a variable $j$ is no longer alive, its value $x_t(j)$ stays unchanged. 
Moreover, $v_t$ will be a random vector with mean $\E[v_t]=0$, and satisfies $\|v_t\|_2=1$ and $x_t \perp v_t$  (and several other properties that we will specify later).
This ensures that $\|x_t\|_2^2=t$ at all steps $t$, and that the process ends by $t=n$.
Also, notice that $n_t \geq n-t$.

For a technical reason, that will be clear later, we only run the process above until $n_t \geq \log^6 n$. Once $n_t$ reaches $\log^6 n$, we will run a single ``rounding" step, using an algorithmic version of Banaszczyk's theorem, incurring an additional $O(\sqrt{k \log n_t}) = O (\sqrt{k \log \log n})$ discrepancy.

\medskip
\noindent{\bf Barrier.} At each time step $t$, we also maintain an upper bound $b_t > 0$, called a {\em barrier}, for the one-sided discrepancy $\max_i \langle a_i, x_t \rangle$. Starting from $b_0$ initially at $t=0$, we will also move the barrier at a speed $c_t\geq 0$, so that $d b_t = c_t d t$.
Crucially, the speed $c_t$ will be chosen to ensure that the overall movement of the barrier satisfies 
\begin{equation}
\label{eq:final-bn}
\int_0^{n} c_t dt \leq b_0 , \quad \text{so that } b_n = b_0 +  \int_{0}^n c_t dt\leq 2 b_0 .
\end{equation}
As we shall see, we will be able to choose $b_0 = O(\sqrt{k \log \log n})$, which will give the result.
The formal definition of the process --- how $v_t$ and $c_t$ are chosen, will be given in \Cref{sec:defn_process}.

\medskip
\noindent \textbf{Handling Large Rows.} 
Even though rows can be arbitrarily large, a standard trick allows us to ignore these until their sizes become $O(k)$.

At any time $t$, we say that a row $i$ is {\em large}, if it has more than $10k$ alive elements,
i.e., $|\{j \in \mathcal{V}_t: a_i(j) \in \{ \pm 1\}\}| > 10k$. Otherwise, we call  it {\em small}.
Note that there are at most $n_t/10$ large rows at any time $t$ (as the matrix $A$ restricted to columns in $\mathcal{V}_t$ has at most $k n_t$ non-zero entries). 
Also, observe that once a row becomes small, it stays small. 

The algorithm will always choose $v_t$ to be orthogonal to each large row at time $t$. This ensures that
$\langle a_i,x_t\rangle =0 $ as long as row $i$ is large.

\medskip
\noindent \textbf{Slack for Rows.} 
To control the discrepancy of rows, we define the following notion of slack at any given time $t$.
Intuitively, this measures how close to the barrier is the discrepancy $\langle a_i, x_t \rangle$, and a smaller slack means a higher risk that the discrepancy of that row might violate the barrier.

Let $\beta =b_0/(20k)$. For a row $i$, we define its slack at time $t$ as 
\begin{equation}
\label{eq:slack}
s_i(t)  := \begin{cases}
    b_t - \langle a_i, x_t \rangle  - \beta  \sum_{j=1}^n a_i(j)^2 (1 - x_t(j)^2), &  \qquad \text{if row $i$ is small at time $t$}\\
      b_0/2, & \qquad \text{if row $i$ is large at time $t$}
\end{cases}
\end{equation} 
The choice of $\beta$ ensures that at time $t=0$,  
\begin{equation}
\label{eq:slack-time0}
s_i(0) \in [b_0/2, b_0].
\end{equation}
This is because, as $x_0=0^n$, we have $s_i(0) = b_0 - \beta \sum_{j=1}^n a_i(j)^2 \geq b_0 - \beta (10k) \geq b_0/2$ for small rows, and $s_i(0) = b_0/2$ for large rows.  Similarly, we have the following.
\begin{observation}
\label{obs:slack-jump}
When some large row $i$ becomes small, the slack $s_i(t)$ can only increase.
\end{observation}
\begin{proof} Let $t_i$ be the time when row $i$ first becomes small. Then $s_i(t) = b_0/2$ for $t<t_i$. At $t=t_i$, the discrepancy $\langle a_i ,x_{t_i}\rangle =0$ as the algorithm ensures that $\langle v_t,a_i\rangle =0$ for all $t<t_i$. So $s_i(t_i) = b_t - \beta \sum_{j=1}^n a_i(j)^2 (1-x_{t_i}(j)^2)  \geq b_t - \beta (10k) =b_t - b_0/2 \geq b_0/2$.
\end{proof}

Our algorithm will always ensure that $x_t(j)^2 \leq 1$, so the last term in \eqref{eq:slack} is always non-negative. Thus, $s_i(t)>0$ implies that row $i$ has discrepancy  $\langle a_i,x_t \rangle < b_t$.
As we also add the row $-a_i$, this implies $|\langle a_i,x_t \rangle| < b_t$, and thus 
at any time $t$, by \eqref{eq:slack} and \eqref{eq:final-bn}
\begin{equation}
    \label{eq:slack-upperbd}
    s_i(t) \leq b_t + |\langle a_i,x_t \rangle| \leq 2 b_t \leq 2 b_n \leq 4 b_0.
\end{equation}

\subsection{The Exponential Barrier Potential}
\label{subsec:exp_barrier_potential}
Set $\lambda := C \log \log n$ for a large enough constant $C > 0$.\footnote{We remark that for simplicity, we make no effort in optimizing the constants in our proof.}
For each row $i \in [n]$, we define its potential, or its {\em weight}, at time $t$ as
\begin{equation}
\label{eq:potential_i}
\Phi_i(t) := 
    \exp \left(\frac{\lambda b_0}{s_i(t)}\right), 
\end{equation}
that  penalizes slacks based on how much smaller they are compared to $b_0$. 
Notice that initially at $t=0$, by \eqref{eq:slack-time0}, we have that
$\Phi_i(0)  \in [\exp(\lambda), \exp(2 \lambda)]$ for each row $i$.

Let $\Phi(t) := \sum_{i=1}^n \Phi_{i}(t)$ denote the total potential of all the rows at time $t$.
Our algorithm will maintain that $\Phi(t) \leq 10 \Phi(0)$  throughout the process, which implies that
\begin{equation}
    \Phi(t) \leq 10 n\exp(2 \lambda).
\end{equation}
As we only use this upper bound, we abuse the notation slightly to denote $\Phi_i(0) = \exp(2 \lambda)$ and $\Phi(0) = n\exp(2 \lambda)$.\footnote{Formally, define $\Phi_i(0) := \exp(2 \lambda)$ and $\Phi(0) := n\exp(2 \lambda)$ before the process starts. Note that $\Phi_i(0^+) \leq \Phi_i(0)$ when the process starts at time $0^+$, which can only decrease the potential (discontinuously).}

\subsection{Evolution of Slack and Potential} 
Let us see how the slack and potential evolve when we update $x_t$ by $v_t \sqrt{dt}$ at time $t$.
By setting $dt = 1/\poly(n)$ arbitrarily small, we can safely ignore lower order terms than $dt$, and we will not write these out to avoid clutter.\footnote{All the coefficients in our computations are bounded by $M =\poly(n)$ and the number of steps is $O(n/dt)$. 
So the total contribution of the $o(dt) = O(M dt^{3/2})$ terms is $O(M ndt^{1/2})$, which is negligible if $dt = o(1/M^2n^2)= 1/\poly(n)$.}  
For each row $i$ and time $t$, it will be convenient to define the vector $e_{t,i}$ with entries \[e_{t,i}(j) := a_i(j)^2 x_t(j).\] 
For a vector $w = (w(1),\ldots,w(n)) \in \R^n$, we abuse the notation and use $w^2$ to denote the  vector $(w(1)^2,\ldots,w(n)^2)$ consisting of the coordinate-wise squared entries. 
\begin{lemma}
\label{lem:slack-change}
  The change in slack $d s_i(t) = s_i (t+dt)  - s_i(t)$ for a small row $i$ at time $t$ is given by  
  \begin{equation} 
  \label{eq:slack-change}
  ds_i(t) = \underbrace{\big(2 \beta \langle e_{t,i} , v_t \rangle -  \langle a_i, v_t \rangle \big) \sqrt{dt}}_{\text{random martingale term}} + \underbrace{\big( c_t + \beta \langle a_i^2 , v_t^2 \rangle \big) d t}_{\text{drift term}}. \end{equation}
  For a large row $i$, as $s_i(t)=b_0/2$ by definition, we have $ds_i(t)=0$.\footnote{Strictly speaking, by Observation \ref{obs:slack-jump} the slack may also increase discontinuously (at most once) when a large row $i$ becomes small, but this only helps.}

\end{lemma}
Notice that the random martingale term has mean zero as $\E[v_t]=0$, and the drift term is always non-negative and helps us. Also, as $dt$ is infinitesimally small, $\sqrt{dt}$ is much larger than $dt$. 
\begin{proof} 
By definition $s_i(t)=b_0/2$ if row is large at time $t$ and thus $ds_i(t)=0$. 
For small rows, using the expression for the slack in \eqref{eq:slack} and as $dx_t=  v_t \sqrt{dt} $ and $db_t = c_t dt$, we have
\[   d s_i(t) 
    = c_t dt - \langle a_i, v_t \sqrt{dt}\rangle  - \beta \sum_{j=1}^n a_i(j)^2\left( \big( x_t(j) + v_t(j) \sqrt{dt}\big)^2 - x_t(j)^2\right).\]
    Rearranging the terms gives
   $ \big(2 \beta \langle e_{t,i} , v_t \rangle -  \langle a_i, v_t \rangle \big) \sqrt{dt} + \Big( c_t + \beta \sum_{j=1}^n a_i(j)^2 v_t(j)^2\Big) d t$.
\end{proof}
\medskip
\noindent {\bf Potential Change.}
Consider the function $f(x) = \exp(\lambda b_0/x)$. Then \[f'(x) = -\frac{\lambda b_0}{x^2} f(x) \quad \text{ and }\quad 
 f''(x) =  \frac{2 \lambda b_0}{x^3} f(x) +  \frac{\lambda^2 b_0^2}{x^4} f(x).\]
As $\Phi_i(t) = f(s_i(t))$, its Ito derivative is given by 
\begin{equation}
\label{eq:phi-derivative}
d\Phi_i(t) = f'(s_i(t)) ds_i(t) + \frac{1}{2} f''(s_i(t)) ds_i(t)^2.
\end{equation}
Plugging the expressions for $f'$ and $f''$ in gives that for a small row $i$,
\begin{align*}
d \Phi_i(t)  &= - \frac{\lambda b_0}{s_i(t)^2} \Phi_i(t) d s_i(t) + \Big(\frac{ \lambda^2 b_0^2}{2 s_i(t)^4} + \frac{\lambda b_0}{s_i(t)^3} \Big) \Phi_i(t) (d s_i(t))^2 \\
& = \frac{\lambda b_0 \Phi_i(t)}{s_i(t)^2}  \Big( - d s_i(t) + \Big(\frac{\lambda b_0}{2 s_i(t)^2} + \frac{1}{s_i(t)} \Big) (d s_i(t))^2  \Big) \\
& \leq \frac{\lambda b_0 \Phi_i(t)}{s_i(t)^2}  \Big( - d s_i(t) + \frac{\lambda b_0}{s_i(t)^2} \cdot (d s_i(t))^2  \Big).
\end{align*}
The inequality in the last line uses that $1/s_i(t) \leq \lambda b_0/(2 s_i(t)^2)$. This holds as $0 < s_i(t) \leq 2 b_t \leq 4 b_0$ by \eqref{eq:slack-upperbd}, and as $4 \ll \lambda$ (since $\lambda = C \log  \log n$). 

To simplify notation, let us define 
\begin{equation}
    \label{def:alpha-gamma}
    \alpha_i(t) := \frac{\lambda b_0}{s_i(t)^2} \quad \text{  and  } \quad   \gamma_i(t) := \frac{\lambda b_0 \Phi_i(t)}{s_i(t)^2} = \alpha_i(t) \Phi_i(t).
\end{equation} Then the above upper bound on $d \Phi_i(t)$ becomes
\begin{align} \label{eq:row_weight_change}
d \Phi_i(t) & \leq \gamma_i(t) \cdot \big ( - d s_i(t) + \alpha_i(t)(d s_i(t))^2 \big)  \nonumber \\
& =  \gamma_i(t) \cdot \Big(\underbrace{ \langle a_i - 2 \beta e_{t,i}, v_t \rangle \sqrt{dt}}_{\text{random martingale term}} -   \underbrace{\big(c_t + \beta \langle a_i^2 , v_t^2 \rangle - \alpha_i(t) \langle 2 \beta e_{t,i} - a_i, v_t \rangle^2 \big) dt}_{\text{drift term}}  \Big) ,
\end{align}
where the second line follows from the expression for $ds_i(t)$ in \eqref{eq:slack-change} and ignoring the $O(dt^{3/2})$ and lower order terms. 

For a large row $i$, either $d\Phi_i(t)=0$ as  $ds_i(t) = 0$, or, by Observation \ref{obs:slack-jump}, the slack may increase discontinuously (once) when it becomes small, which only decreases the potential deterministically.

\medskip
\noindent \textbf{Our Strategy.}
Having computed the relevant quantities, let us consider our high level strategy.
Our goal is to find a direction $v_t$ so that the expected total potential change $\E[d\Phi_t ] \leq 0$.
As $\E[v_t]=0$, for each small row $i$, it follows from \eqref{eq:row_weight_change} that 
\begin{align} \label{eq:potential_change_1}
\E[d \Phi_i(t)] \leq  - \gamma_i(t) \cdot \Big( \underbrace{c_t + \beta \cdot \E \langle a_i^2 , v_t^2 \rangle}_{\text{these terms help us}} - \underbrace{\alpha_i(t) \cdot \E \langle 2 \beta e_{t,i} - a_i, v_t \rangle^2 }_{\text{this term hurts us}} \Big) dt .
\end{align}
As the scaling factors $\gamma_i(t)$ in \eqref{eq:potential_change_1} for each row $i$ can be potentially arbitrary, we will try to ensure that $\E[d \Phi_i(t)] \leq 0$ for each row $i$. That is, the terms that help us in \eqref{eq:potential_change_1} need to be at least as large as the term that hurts us.
Of course, we want to do this while keeping $c_t$ as small as possible (as this determines our final discrepancy).

We can ensure for free that $v_t$ looks random-like (see Theorem \ref{thm:sub-isotropic-SDP} below), so the term that hurts us is roughly $\alpha_i(t) \cdot \E\langle a_i^2,v_t^2\rangle$ (let's ignore the term $2\beta e_{t,i}$ for now). So we basically need to ensure \[c_t + (\beta - \alpha_i(t)) \, \E\langle a_i^2,v_t^2 \rangle \geq 0.\]
If $\alpha_i(t) \ll \beta$, this is trivially satisfied, and thus we refer to such rows as {\em safe}.  It will turn out that this the case for most of the rows. 

Now, consider the {\em dangerous} rows with  $\alpha_i(t) \gg \beta$. Looking at \eqref{def:alpha-gamma},  $\alpha_i(t)$ is large only for rows with small slack $s_i$ and hence high weights $\Phi_i(t)$. The number of such rows will be relatively small (about $n/\polylog(n)$ by our choice of parameters later), as we have the control $\Phi(t) \leq 10 \Phi(0)$ on the total potential. For such dangerous rows, we will show that with high probability, at any time $t$, their average size is much smaller than $k$. This would imply that $\E\langle a_i^2,v_t^2 \rangle$ is small for such rows (on average) and $c_t$ can be moved slowly. Note that as this guarantee is only on average, some high weight rows could still be too large. However, we will ensure that there are only $O(n_t)$ such ``high-weight and large size" rows. For such rows, we can simply set $v_t$ orthogonal to their $ 2 \beta e_{i,t}-a_i$ so that $d\Phi_i(t)\leq 0$ (deterministically) in \eqref{eq:row_weight_change}.

\medskip
\noindent \textbf{Safe and Dangerous Rows.} 
We now formalize the notions of {\em safe} and {\em dangerous} rows. 
Recall from Section \ref{subsec:exp_barrier_potential} that at time $0$, each row $i$ has weight at most $\Phi_i(0) = e^{2 \lambda}$.
\begin{definition}[Safe and dangerous rows.]
Call a row $i$ at time $t$ {\em safe} if its weight is $\Phi_i(t) \leq \Phi_i(0) e^\lambda = e^{3\lambda}$, and {\em dangerous} otherwise. 
Equivalently, row $i$ is safe at time $t$ if $s_i(t) \geq b_0 / 3$. 
We denote by $\mathcal{R}_{\safe}(t)$ (resp. $\mathcal{R}_{\dang}(t)$) the set of safe (resp. dangerous) rows at time $t$. 
\end{definition}
Note that a large row $i$ is always safe as $s_i(t) = b_0/2$ and hence $\Phi_i(t) = \Phi_i(0)$ (recall our convention that $\Phi_i(0) = \exp(2 \lambda)$), and by \Cref{obs:slack-jump}, it is also safe when it just becomes small.

Let us also make the simple but useful observation that if we can ensure that $\Phi_t \leq 10 \Phi(0)$, then the number of dangerous rows   $|\mathcal{R}_{\dang}(t) | \leq 10n e^{-\lambda} = 10n/(\log n)^C$ by our choice of $\lambda = C \log \log n$.

\medskip
\noindent {\bf The Matrix $E(t)$.} 
At time $t$, define $E(t) \in \R^{n \times n_t}$ to be the matrix whose $i$th row is $2 \beta e_{t,i} - a_i$ restricted to coordinates in $\mathcal{V}_t$. Note that the entries of $E(t)$ are bounded in magnitude by $1.1$, as $\beta = b_0/(20k) = O(\sqrt{(\log \log n)/k}) =o(1)$ and $\|e_{t,i}\|_\infty \leq 1$ (recall that $e_{t,i}(j) = a_i(j)^2 x_t(j)$).

This matrix will play an important role, as its entries correspond to the random term in \eqref{eq:row_weight_change} and the only positive term (that hurts us) in \eqref{eq:potential_change_1}. 

We use $E_{\safe}(t) \in \R^{|\mathcal{R}_{\safe}(t)| \times n_t}$ (resp. $E_{\dang}(t) \in \R^{|\mathcal{R}_{\dang}(t)| \times n_t}$) to denote the restriction of $E(t)$ to the safe rows $\mathcal{R}_{\safe}(t)$ (resp.~dangerous rows $\mathcal{R}_{\dang}(t)$).

\subsection{Formal Definition of the Process}
\label{sec:defn_process}
We now define our process formally and specify how to choose the random vector $v_t$ at time $t$.

We will need the following known result about finding sub-isotropic random vectors orthogonal to a subspace. This was first proved in \cite{BG17}, but  the statement below is from \cite[Thm 2.5]{Ban24}. 
\begin{theorem}[\cite{BG17}] \label{thm:sub-isotropic-SDP}
Let $W \subset \R^h$ be a subspace with dimension $\dim(W) = \delta h$. Then for any $0 < \kappa, \eta < 1$ such that $\eta + \kappa \leq 1-\delta$, there is a $h \times h$ PSD matrix $U$ satisfying the following:

(i) $\langle ww^\top, U \rangle = 0$ for all $w \in W$, 

(ii) $U_{ii} \leq 1$ for all $i \in [h]$, 

(iii) $\Tr(U) \geq \kappa h$, and 

(iv) $U \preceq \frac{1}{\eta} \diag(U)$. 

Furthermore, such a PSD matrix $U$ can be computed by solving a semi-definite program (SDP). 
\end{theorem}

\smallskip
\noindent \textbf{Choosing $v_t$.} As mentioned earlier, $v_t = 0$ if the number of alive elements $n_t < \log^6 n$. We specify below how to choose $v_t$ when $n_t \geq \log^6 n$.

For our purposes, we will apply Theorem \ref{thm:sub-isotropic-SDP} to the set of alive elements $\mathcal{V}_t$ (so the ambient dimension  $h=n_t$) with parameters $\delta \leq 1/2$, $\eta = 1/4$ and $\kappa = 1/4$. 

We will specify how to choose the subspace $W_t \subset \R^{n_t}$ with dimension $\dim(W_t) \leq n_t/2$ at time $t$ below. 
Once we fix $W_t$, we will first apply \Cref{thm:sub-isotropic-SDP} to obtain the PSD matrix $U_t$ such that $\Tr(U_t) \geq n_t/4$. Let $U_t = Q \Lambda_t Q^\top$ be its spectral decomposition. Then we choose the random vector 
\begin{align*}
    v_t := \frac{1}{\sqrt{\Tr(U_t)}} U_t^{1/2} Q r_t = \frac{1}{\sqrt{\Tr(U_t)}} Q \Lambda_t^{1/2} r_t ,
\end{align*}
where $r_t \in \R^{n_t}$ is a random vector whose coordinates are independent Rademacher random variables (taking values $1$ or $-1$ with probability $1/2$ each). Note that the choice of $v_t$ above satisfies that 
\begin{align} \label{eq:unit_vt}
\|v_t\|_2^2 = v_t^\top v_t = \frac{1}{\Tr(U_t)} r_t^\top \Lambda_t^{1/2} Q^\top Q \Lambda_t^{1/2} r_t = \frac{\Tr(\Lambda_t)}{\Tr(U_t)} = 1 .
\end{align}
This ensures that the update $v_t\sqrt{dt}$ to the coloring vector $x_t$ has $\ell_2$ norm exactly $\sqrt{d t}$. 

\medskip
\noindent \textbf{Choosing the Subspace $W_t$.}
The subspace $W_t$ will consist of the linear span of several vectors that we specify below.   
 These vectors will consist of certain rows that we want to 
 {\em block}, as well as the largest (right) singular vectors of $E_{\dang}(t)$ and $E_{\safe}(t)$. 

\medskip
\noindent \textbf{Blocking.} 
We will block the following rows:
\begin{enumerate}
    \item {\bf Large rows.} For each large row $i$, i.e., $|\{j \in \mathcal{V}_t: a_i(j) \in \{\pm 1\}\}| > 10 k$, add the vector $a_i$ (restricted to $\mathcal{V}_t$) to $W_t$. There can be at most $n_t/10$ such rows.
    \item {\bf Small slacks.} For the $n_t/10$ rows $i$ with the largest potential $\Phi_i(t)$ (or equivalently, the smallest slack $s_i(t)$) add the vectors $2\beta e_{t,i} - a_i$ to $W_t$. 
    \item {\bf Dangerous rows with many elements.} For the $n_t/10$ dangerous rows $i$ in $\mathcal{R}_{\dang}(t)$ with the most number of alive elements,   add the vectors $2\beta e_{t,i} - a_i$ to $W_t$.
\end{enumerate}
Let $\mathcal{B}_t \subseteq [n]$ denote the set of rows that are blocked above. In addition to the rows above, we also add the (single) vector $x_t$ to $W_t$. This will ensure that $v_t$ is orthogonal to $x_t$ and  that $\|x_t\|_2^2$ is exactly $t$ (as long as $n_t \geq \log^6 n$). 

\medskip
\noindent \textbf{Blocking Large Singular Vectors of $E_{\safe}(t)$ and $E_{\dang}(t)$.} We also add (crucially) the following constraints corresponding to the largest singular values of $E_{\safe}(t)$ and $E_{\dang}(t)$.
Consider the singular value decomposition  \[E_{\dang}(t) = L_{\dang}(t) \Sigma_{\dang}(t) R_{\dang}(t)^\top \text{ and } E_{\safe}(t) = L_{\safe}(t) \Sigma_{\safe}(t) R_{\safe}(t)^\top,\] of  
$E_{\dang}(t)$  and $E_{\safe}(t)$, 
where the columns of $L_{\dang}(t), L_{\safe}(t) \in \R^{n \times n}$ and $R_{\dang}(t), R_{\safe}(t) \in \R^{n_t \times n_t}$ are the left and right singular vectors respectively, of $E_{\dang}(t)$ and $E_{\safe}(t)$. 

We add to $W_t$ the $n_t/11$ right singular vectors in $R_{\dang}(t)$ corresponding to the largest $n_t/11$ singular values of $E_{\dang}(t)$, as well as the $n_t/11$ right singular vectors in $R_{\safe}(t)$ corresponding to the largest $n_t/11$ singular values of $E_{\safe}(t)$. 

Notice that  the total number of vectors added to $W_t$ is at most $ 3 n_t/10 + 1 + 2n_t/11 \leq n_t/2$, and thus $\dim(W_t) \leq n_t/2$.

\medskip
\noindent \textbf{Barrier Movement.} When $n_t \geq \log^6 n$, we move the barrier $b_t$ at rate 
\begin{align} \label{eq:c_t_setting}
c_t := O \Big( \frac{\lambda k}{b_0 n_t \log n} \Big) ,
\end{align}
where we choose the constant to be sufficiently large. As discussed earlier, $c_t = 0$ if $n_t < \log^6 n$.  

\medskip
\noindent {\bf Some Consequences of Blocking.} Let us note two direct consequences of blocking.

\begin{lemma}[Slack lower bound for unblocked rows] \label{lem:slack_lb_unblocked}
At any time $t$ satisfying $\Phi(t) \leq 10 \Phi(0)$, 
every unblocked row $i \notin \mathcal{B}_t$ has slack 
\[
s_i(t) \geq \frac{\lambda b_0}{2 \lambda + \log (100 n/n_t)} .
\]
\end{lemma}
\begin{proof}
As we block the $n_{t}/10$ rows with the largest potential and $\Phi(t) \leq 10 \Phi(0)$, 
any row with potential more than $ 100\Phi(0)/{n_{t}}$ must be blocked. Thus for any unblocked row $i$, 
\[
\exp\left(\frac{\lambda b_0}{s_i(t)} \right) \leq \frac{100 n}{n_t}  \exp(2 \lambda) = \exp(2 \lambda + \log(100n/n_t)).
\]
Taking logarithms and rearranging gives the claimed result. 
\end{proof}

\begin{lemma}[Covariance bound] \label{lem:cov_bound}
For any time $t$, 
we have $\E[v_t v_t^\top] \preceq O\big(\frac{1}{n_t} \big) \cdot I$, and that 
\[
\E \big[E_{\dang}(t) v_t v_t^\top E_{\dang}(t)^\top \big] \preceq O\Big(\frac{k}{n_t} \Big) \cdot I \quad \text{and} \quad \E \big[E_{\safe}(t) v_t v_t^\top E_{\safe}(t)^\top \big] \preceq O\Big(\frac{k}{n_t} \Big) \cdot I  . 
\]
\end{lemma}

\begin{proof}
For the first statement, note that
\begin{align*}
\E[v_t v_t^\top] = \frac{1}{\Tr(U_t)} U_t \preceq \frac{1}{\kappa n_t} \cdot \frac{1}{\eta} \diag(U_t) \preceq O\Big(\frac{1}{n_t}\Big) \cdot I ,
\end{align*}
where the inequalities above use the properties $\Tr(U_t) \geq \kappa n_t$, $U_t \preceq (1/\eta) \, \diag(U_t)$, and each diagonal entry $U_t(i,i) \leq 1$ guaranteed by \Cref{thm:sub-isotropic-SDP}. 

To prove the second statement, we denote $\sigma_{1}(t) \geq \cdots \geq \sigma_{n_t}(t) \geq 0$ the singular values of $E_{\dang}(t)$.
Since each of the $n_t$ columns of $E_{\dang}(t)$ has at most  
$k$ non-zero entries by definition, and each entry of $2\beta e_{t,i}(j) - a_i(j)$ has magnitude at most $1.1$, we have 
\[
\sum_{i=1}^{n_t} \sigma_i(t)^2 = \Tr(E_{\dang}(t) E_{\dang}(t)^\top) \leq (1.1)^2 n_t k \leq 1.5 n_t k .
\]
This implies that $\sigma_{n_t/11}(t) \leq \sqrt{20 k}$.  
As $v_t$ is orthogonal to the largest $n_t/11$ right singular vectors in $R_{\dang}(t)$, denoting $\widetilde{E}_{\dang}(t)$ as the modification of $E_{\dang}(t)$ with the largest $n_t/11$ singular values set to $0$, 
\begin{align*}
\E \big[E_{\dang}(t) v_t v_t^\top E_{\dang}(t)^\top \big]
& = \E \Big[ \widetilde{E}_{\dang}(t) v_t v_t^\top \widetilde{E}_{\dang}(t)^\top\Big] \\
& \preceq O \Big(\frac{1}{n_t}\Big) \cdot \E\Big[\widetilde{E}_{\dang}(t) \widetilde{E}_{\dang}(t)^\top\Big] \preceq O\Big(\frac{k}{n_t} \Big) \cdot I .
\end{align*}
The third statement can be proved similarly, by noticing that $\Tr(E_{\safe}(t) E_{\safe}(t)^\top) \leq 1.5 n_t k$. The rest of the argument is analogous. 
\end{proof}

\subsection{Discrepancy Bound}
Let first see 
how the setting of $c_t$ in \eqref{eq:c_t_setting},
gives the claimed discrepancy bound in Theorem \ref{thm:beck-fiala-main}.
\begin{lemma}[Discrepancy bound] \label{lem:disc_bound}
For the setting of $c_t$ as in \eqref{eq:c_t_setting}, $b_{n}$ is bounded by 
\[
b_{n} \leq b_0 + O\big(\lambda/b_0 \big) .
\]
Therefore, setting $b_0 = O(\sqrt{\lambda k})$ for a sufficiently large constant guarantees that 
\[
b_{n} \leq 2 b_0 = O(\sqrt{\lambda k}) = O(\sqrt{k \log \log n}).
\]
\end{lemma}
\begin{proof}
Let $t_{f}$ be the earliest time when $n_t \leq \log^6 n$ and the process is frozen.
As $\|v_t\|_2 = 1$ and $\langle v_t,x_t\rangle =0 $, we have $\|x_t\|_2^2 = t$ for all $t \leq t_f$. 
Therefore, $t \geq (n-n_t) (1-1/2n)^2 \geq n-n_t-1$ for all $t\leq t_f$, 
and we have that $n_t \geq \max(n-t-1, \log^6 n)$ for all $t \leq t_f$. As $c_t=0$ for $t\in [t_f,n]$, 
we can upper bound $b_n$ as
\begin{align*}
b_n 
& = b_0 + \int_0^{t_f} c_t d t \leq  b_0 + O\big(\frac{\lambda k}{b_0 \log n} \big) \cdot \int_0^{n} \frac{1}{\max(n-t-1, \log^6 n)}  dt 
 = b_0 + O\big(\frac{\lambda k}{b_0} \big) .
\end{align*}
The second statement of the lemma immediately follows from the first. 
\end{proof}

\smallskip
\noindent \textbf{Rounding to Full Coloring.} 
Note that the process in \Cref{sec:defn_process} only produces a fractional $x_n \in [-1,1]^n$ at time $n$. But this fractional coloring can be easily rounded to a full coloring in $\{\pm 1\}^n$ without incurring too much additional discrepancy. 

\begin{lemma}[Rounding to full coloring] \label{lem:round_full_coloring}
Let $x_n \in [-1,1]^n$ be the fractional coloring at time $n$ in the process defined in \Cref{sec:defn_process}. Then there exists a full coloring $\widetilde{x} \in \{\pm 1\}^n$ such that for each row $i \in [n]$, we have
\[
    \big|\langle a_i, \widetilde{x} - x_n \rangle \big| \leq O(\sqrt{k \log \log n}) .
\]
\end{lemma}

\begin{proof} 
The coloring $x_n$ has at most  $\log^6 n$ elements in $\mathcal{V}_n$, i.e., $|\{j \in [n]: |x_n(j)| \leq 1- 1/2n\}|$. 
For each element $j \in [n] \setminus \mathcal{V}_n$ we can simply round $x_n(j)$ to the closest $\pm 1$. This affects the discrepancy of each row $i$ by at most $1$. 
For elements in $\mathcal{V}_n$, as $|\mathcal{V}_n| < \log^6 n$, we can find a full coloring $\widetilde{x}_{\mathcal{V}} \in \{\pm 1\}^{\mathcal{V}_n}$ using an algorithmic version of Banaszczyk's theorem (e.g., \cite{BDG19}),
such that for each row $i \in [n]$, 
\[
\Big| \sum_{j \in \mathcal{V}_n} a_i(j) \big(\widetilde{x}_{\mathcal{V}}(j) - x_n(j) \big) \Big| \leq O(\sqrt{k \log \log n}).\qedhere
\]
\end{proof}

\Cref{lem:disc_bound,lem:round_full_coloring} imply a total discrepancy bound of $O(\sqrt{k \log \log n})$, as claimed in  \Cref{thm:beck-fiala-main}.  To finish the proof of \Cref{thm:beck-fiala-main}, it thus remains to prove that the potential stays bounded. 

\subsection{Controlling Safe Rows}

We first show that the safe rows 
are handled easily.
In particular, we have the following lemma.
\begin{lemma}[Controlling safe rows] \label{lem:safe_rows_energy}
Suppose $b_0 \geq 50 \sqrt{\lambda k}$. For any safe row $i$ at time $t$, the random vector $v_t$ as defined in \Cref{sec:defn_process} satisfies that
\begin{align*}
\alpha_i(t) \cdot \E\langle 2 \beta e_{t,i} - a_i , v_t \rangle^2 \leq \frac{1}{2} \beta \cdot \E \langle a_i^2, v_t^2 \rangle. 
\end{align*}
\end{lemma}
Using this bound in \eqref{eq:potential_change_1}, immediately gives that $\E[d \Phi_i(t)]<0$  for a safe row $i$ at time $t$.

Notice that as we set $b_0 = O(\sqrt{k \lambda})$ in \Cref{lem:disc_bound},
the assumption in \Cref{lem:safe_rows_energy} is satisfied as $p$ is a constant. 

\begin{proof}
By definition, a safe row $i$ satisfies that $\Phi_i(t) \leq \exp(3\lambda)$. So,
$\exp(\lambda  b_0/s_i(t)) \leq \exp(3 \lambda)$ and hence $s_i(t) \geq b_0/3$.
Consequently, for any safe row $i$ at time $t$, 
\begin{align*}
\alpha_i (t) = \frac{\lambda b_0}{s_i(t)^2} \leq \frac{9 \lambda}{b_0} \leq \frac{b_0}{200k} = \frac{\beta}{10} .
\end{align*}
By \Cref{thm:sub-isotropic-SDP} (with $\eta = 1/4$) and the definition of $v_t$ in \Cref{sec:defn_process}, and as each coordinate of $2 \beta e_{t,i} - a_i$ has magnitude at most $1.1$ 
\begin{align*}
\E\langle 2 \beta e_{t,i} - a_i , v_t \rangle^2 \leq 4 \cdot \E \langle (2 \beta e_{t,i} - a_i)^2, v_t^2 \rangle \leq 5 \cdot \E \langle a_i^2 , v_t^2 \rangle .
\end{align*}
Combing this bound with that for $\alpha_i(t)$  above proves the lemma.
\end{proof}

\subsection{Controlling Dangerous Rows I: Column Weight Concentration}
For dangerous rows, the above argument using energy term fails, as their slacks $s_i$'s can be as small as $\Theta(b_0/\log (2n/n_t))$. 
Instead, we control such rows using the drift term $c_t$. To do so while keeping $c_t$ small, the key idea is control the {\em sizes} of dangerous rows to be $O(k/\poly(\log n))$. We do this by controlling the column weights defined as follows. 

\noindent \textbf{Column Weights.} For an alive column $j \in \mathcal{V}_t$, we let $\mathcal{C}_j$ be the set of non-zeros in that column, i.e., $\mathcal{C}_j := \{i \in [n]: a_i(j) \in \{\pm 1\}\}$. We define the {\em weight of column $j$} to be 
\[
W_j(t) := \sum_{i \in \mathcal{C}_j} \min\{\Phi_i(t),  e^{3\lambda} \} =: \sum_{i \in \mathcal{C}_j} \trunc_{\lambda} (\Phi_i(t)) .
\]
This quantity measures the total weight of $\mathcal{C}_j$ in column $j$ when every row weight is truncated at $e^{3\lambda}$ ---  the precise threshold for a row to be dangerous. As we will see later in \Cref{lem:row_size_via_column_weight} and \Cref{cor:dang_row_size}, column weights can be used to control the degree of the set system restricted to dangerous rows, which directly implies an upper bound on their sizes. 

\medskip
\noindent \textbf{Concentration of Column Weights.}
Notice that initially at $t = 0$, every row is safe, and thus
\[
W_j(0) = k \Phi_i(0) = k e^{2 \lambda} .
\] 
We will show that the column weights $W_j(t)$ do not go much beyond $W_j(0)$ with high probability for all time $t \in [0,n]$ and for all columns $j \in \mathcal{V}_t$. This is where the blocking of large singular vectors of $E_\dang(t)$ and $E_\safe(t)$ will be crucially used.

\begin{lemma}[Concentration of column weights] \label{lem:conc_col_weight}
For the process defined in \Cref{sec:defn_process}, let $\tau_{\bad} > 0$ be the first time $t$ such that $\Phi(t) > 10 \Phi(0)$, and it is defined to be $n$ if $\Phi(t) \leq 10 \Phi(0)$ holds for all $t \in [0,n]$.
Then with high probability, for all time steps $t < \tau_{\bad}$ and all columns $j \in \mathcal{V}_t$, we have
\[
W_j(t) \leq k e^{2\lambda} + O( e^{3\lambda} \log^2 n).
\] 
\end{lemma}
For the proof, we make use of the following inequality for supermartingales. 

\begin{fact}[Lemma 2.2 in \cite{Ban24}]  \label{fact:freedman_conc}
Let $\{Z_k: k = 0,1, \cdots\}$ be a sequence of random variables with $Y_k := Z_k - Z_{k-1}$, such that $Z_0$ is deterministic and $Y_k \leq 1$ for all $k \geq 1$. If for all $k \geq 1$, 
\[
\E_{k-1}[Y_k] \leq - \delta \E_{k-1}[Y_k^2]
\]
holds with $0 < \delta < 1$, where $\E_{k-1}[\cdot]$ denotes $\E[\cdot | Z_1, \cdots, Z_{k-1}]$. Then for all $\xi \geq 0$, it holds that
\begin{align*}
\p\big( Z_k - Z_0 > \xi \big) \leq \exp(- \delta \xi). 
\end{align*}
\end{fact}

\begin{proof}[Proof of \Cref{lem:conc_col_weight}]
We view $\tau_{\bad}$ as a stopping time, and consider the modified process so that $W_j(t)$ remain fixed for all $t \geq \tau_{\bad}$ and $j \in \mathcal{V}_t$. 
Clearly, it suffices to prove the bound in the lemma for the modified process for all times $t$ and columns $j$, and
abusing notation slightly, let $W_j(t)$ denote this modified process. 

Fix a column $j \in [n]$ and a time $t < \tau_{\bad}$ such that $j \in \mathcal{V}_t$. Let us see how its truncated column weight $W_j(t)$ evolves.
The rows in $\mathcal{C}_j$ at time $t$ can be divided into three types, and we consider the contribution of each type separately:
\begin{enumerate}
    \item [(i)]
    For dangerous rows $i \in \mathcal{C}_j \cap \mathcal{R}_{\dang}(t)$,  their truncated contribution $\trunc_\lambda(\Phi_i(t))$ to $W_j(t)$ remains unchanged.\footnote{To be completely rigorous, $W_j(t)$ may decrease when a dangerous row $i \in \mathcal{C}_j$ becomes safe. 
    But this only helps us, as we only need to upper bound $W_j(t)$.}  
    Similarly, for large rows $i$, the slack $s_i(t)$ and hence $\Phi_i(t)$ stay unchanged (or the potential only decreases deterministically once, when the row becomes small). We denote the set of such rows as $\mathcal{C}_{j, 1}(t) \subseteq \mathcal{C}_j$. 
    
    \item [(ii)] Each small row $i \in \mathcal{C}_j \setminus \mathcal{C}_{j,1}(t)$ that is blocked by the process (i.e., $i \in \mathcal{B}_t$ and $\langle 2 \beta e_{t,i} - a_i, v_t \rangle = 0$), by \eqref{eq:row_weight_change}, 
 only contribute a negative drift to the change in $W_j(t)$. We denote the set of such rows as $\mathcal{C}_{j,2}(t) \subseteq \mathcal{C}_j \setminus \mathcal{C}_{j, 1}(t)$. 
    
    \item [(iii)] Denote the remaining rows by $\mathcal{C}_{j, 3}(t) := \mathcal{C}_j \setminus (\mathcal{C}_{j, 1}(t) \cup \mathcal{C}_{j,2}(t))$. The contribution $\trunc_\lambda(\Phi_i(t))$ from such rows contain both the random martingale and drift terms.
\end{enumerate}
\smallskip
\noindent \textbf{Bounding $dW_j(t)$ and Its Moments.} Putting this together, by \eqref{eq:row_weight_change},\footnote{Notice that when a safe row $i$ becomes dangerous, the upper bound in \eqref{eq:row_weight_change} still holds due to the truncation, i.e., $d \trunc_{\lambda}(\Phi_i(t)) \leq d \Phi_i(t)$ in this case.} 
\begin{align}
d W_j(t) &\leq \sum_{i \in \mathcal{C}_{j, 3}(t)} \gamma_i(t) \langle a_i - 2 \beta e_{t,i}, v_t \rangle \sqrt{dt}  \nonumber \\
&\qquad  -\sum_{i \in \mathcal{C}_{j,2}(t) \cup \mathcal{C}_{j,3}(t)} \gamma_i(t) \cdot \Big( c_t + \beta \cdot \langle a_i^2 , v_t^2 \rangle - \alpha_i(t) \cdot \langle 2 \beta e_{t,i} - a_i, v_t \rangle^2 \Big) dt. \label{eq:dwjt}
\end{align}
So the first moment of $d W_j(t)$ can be bounded as 
\begin{align} \label{eq:col_weight_change_drift}
\E[d W_j(t)] &\leq - \sum_{i \in \mathcal{C}_{j,2}(t) \cup \mathcal{C}_{j,3}(t)} \gamma_i(t) \cdot \Big( c_t + \beta \cdot \E \langle a_i^2 , v_t^2 \rangle - \alpha_i(t) \cdot \E \langle 2 \beta e_{t,i} - a_i, v_t \rangle^2 \Big) dt\nonumber  \\
& 
\leq - \sum_{i \in \mathcal{C}_{j,2}(t) \cup \mathcal{C}_{j,3}(t)} \gamma_i(t) \cdot  c_t  dt \leq \sum_{i \in \mathcal{C}_{j,3}(t)} \gamma_i(t) \cdot  c_t  dt ,
\end{align}
where the second inequality follows from \Cref{lem:safe_rows_energy} and as $\mathcal{C}_{j,2}(t) \cup \mathcal{C}_{j,3}(t)$ only contains safe rows.

To bound the second moment of $d W_j(t)$, let $\gamma_{t,j} \in \R^{\mathcal{R}_{\safe}(t)}$ be the column vector defined as 
\[
\gamma_{t,j}(i) := \begin{cases}
\gamma_i(t) \quad &\text{if } i \in \mathcal{C}_{j,3}(t) ,\\
0 \quad &\text{otherwise} .
\end{cases}
\]
Then, using \eqref{eq:dwjt} and ignoring the $o(dt)$ terms,
and as $\mathcal{C}_{j,3}(t)$ contains only safe rows,
the second moment of $d W_j(t)$ can be bounded as 
\begin{align} \label{eq:col_weight_change_martingale}
\E[(d W_j(t))^2] & \leq \E \Big( \sum_{i \in \mathcal{C}_{j, 3}(t)} \gamma_i(t) \langle a_i - 2 \beta e_{t,i}, v_t \rangle\Big)^2 d t \nonumber \\
& = \gamma_{t,j}^\top \E\big[E_{\safe}(t) v_t v_t^\top E_{\safe}(t)^\top \big] \gamma_{t,j} \leq O\Big(\frac{k }{n_t} \Big) \sum_{i \in \mathcal{C}_{j,3}(t)} \gamma_i(t)^2 ,
\end{align}
where the last inequality follows from \Cref{lem:cov_bound}. 

As all $i \in \mathcal{C}_{j,3}(t)$ are safe, their slacks $s_i(t) \geq \frac{b_0}{3}$ and $\Phi_i(t) \leq e^\lambda \Phi_i(0) = e^{3\lambda}$. So 
\begin{align*} 
\gamma_i(t) = \frac{\lambda b_0 \Phi_i(t)}{s_i(t)^2} \leq O(1) \cdot \frac{\lambda e^{3\lambda}}{b_0} .
\end{align*}
\medskip
\noindent \textbf{Applying Concentration Inequality.}
Plugging this bound for $\gamma_i(t)$ in \eqref{eq:col_weight_change_martingale} above gives,
\begin{align*}
\E[(d W_j(t))^2]  
& \leq O(1) \cdot \sum_{i \in \mathcal{C}_{j,3}(t)} \frac{k} {n_t} \cdot \gamma_i(t) \cdot \frac{\lambda  e^{3\lambda}}{b_0 } \nonumber \\
& 
\leq \sum_{i \in \mathcal{C}_{j,3}(t)}  c_t \cdot \gamma_i(t)  \cdot \big( e^{3\lambda} \log n \big) \tag{as $ c_t = O(\lambda k/(b_0n_t \log n))$ } \\
& 
 \leq \big(e^{3\lambda} \log n \big) \cdot \big(- \E[d W_j(t)] \big) \tag{by $\eqref{eq:col_weight_change_drift} $}. 
\end{align*}
 In other words, $W_j(t)$ satisfies the condition in \Cref{fact:freedman_conc}  with $\delta := (\Phi_i(0) e^\lambda \log n)^{-1}$.
 Observe that this condition is also trivially satisfied when the modified process $W_j(t)$ freezes after time step $\tau_{\bad}$ (as both $\E[dW_j(t)]= \E[dW_j(t)^2]=0$).

 As $W_j(0) = k\Phi_i(0) = k e^{2\lambda}$,  choosing $\xi := O( e^{3\lambda} \log^2 n)$ so that $\delta \xi = O(\log n)$, \Cref{fact:freedman_conc}  gives that 
\[ \Pr[ W_j(t) - k e^{2\lambda} \geq \xi] \leq \exp(-\delta  \xi) =  \exp(-O(\log n)).\]
The lemma now follows by choosing the constant $O(\cdot )$ in $\xi$ large enough, and applying a union bound over all $\poly(n)$ time steps $t$ and the columns $j$.     
\end{proof}

\subsection{Controlling Dangerous Rows II: Bounded Row Sizes}
The control on $W_j(t)$ above directly implies an upper bound on the sizes of dangerous rows. 

\begin{lemma}[Column weight controls row sizes] \label{lem:row_size_via_column_weight}
For any time $t \in [0,n]$ and $j \in \mathcal{V}_t$, we have $|\mathcal{C}_j \cap \mathcal{R}_{\dang}(t)| \leq W_j(t)/e^{3\lambda}$. Consequently, except for at most $n_t/10$  rows in $\mathcal{R}_{\dang}(t)$, the number of alive elements in any other dangerous row is at most $\max_{j \in \mathcal{V}_t} 10 W_j(t)/ e^{3\lambda}$. 
\end{lemma}

\begin{proof}
The first statement follows as each row in $\mathcal{C}_j \cap \mathcal{R}_{\dang}(t)$ contributes exactly $e^{3\lambda}$ to $W_j(t)$. 

For the second statement,
as each column $j\in \mathcal{V}_t$ has at most $\max_{j \in \mathcal{V}_t} W_j(t)/ e^{3\lambda}$ non-zero entries in the dangerous rows $\mathcal{R}_{\dang}(t)$, there are at most $\max_{j \in \mathcal{V}_t} n_t W_j(t)/e^{3\lambda}$ such entries in the columns in $\mathcal{V}_t$. 
Therefore, excluding the $n_t/10$ rows in  $\mathcal{R}_{\dang}$ with the most number of non-zeros in $\mathcal{V}_t$, every other dangerous row can have at most 
\[
\max_{j \in \mathcal{V}_t}\frac{n_t W_j(t) / e^{3\lambda}}{n_t/10} = \max_{j \in \mathcal{V}_t} \frac{10 W_j(t)}{e^{3\lambda}} 
\]
non-zeros among the alive columns $\mathcal{V}_t$.  
\end{proof}

As a corollary, the following bound on the sizes of dangerous rows can be deduced. 

\begin{corollary}[Dangerous row sizes] \label{cor:dang_row_size}
For the process defined in \Cref{sec:defn_process}, if the statement of \Cref{lem:conc_col_weight} holds 
at time $t$, then every dangerous and unblocked row $i \in \mathcal{R}_{\dang}(t) \setminus \mathcal{B}_t$ has size 
\[
\Big|\big\{ j \in \mathcal{V}_t: a_i(j) \in \{\pm 1\} \big\}\Big| \leq \frac{10 k}{e^\lambda} + O(\log^2 n) .\]
\end{corollary}
\begin{proof}
As the statement of \Cref{lem:conc_col_weight} holds at time $t$ for all $j \in \mathcal{V}_t$, we have
\[
W_j(t) \leq k \Phi_i(0) + O( \Phi_i(0) e^\lambda \log^2 n).
\]
As the process defined in  \Cref{sec:defn_process} blocks the $n_t/10$ rows in $\mathcal{R}_{\dang}(t)$ with the highest number of alive elements, the claimed bound follows immediately from \Cref{lem:row_size_via_column_weight}. 
\end{proof}

Using the bound on row sizes in \Cref{cor:dang_row_size}, we can now show that the potential decreases in expectation for dangerous rows. 

\begin{lemma}[Controlling dangerous rows via $c_t$] \label{lem:dang_rows_ct}
Assume $\lambda \geq 3 \log \log n$ and $k \geq \log^5 n$. Also assume $\Phi(t) \leq 10 \Phi(0)$, and that the statement of \Cref{cor:dang_row_size} holds at time $t \in [0,n]$. Then for any dangerous and unblocked row $i \in \mathcal{R}_{\dang}(t) \setminus \mathcal{B}_t$, we have
\begin{align*}
    \alpha_i(t) \cdot \E \langle 2 \beta e_{t,i} - a_i, v_t \rangle^2 \leq \frac{c_t}{2}. 
\end{align*}
Consequently, it follows from \eqref{eq:potential_change_1} that $\E[d \Phi_i(t)] \leq - \frac{1}{2} \gamma_i(t) c_t$. 
\end{lemma}
\begin{proof}
By assumptions in the lemma, the support size of $2 \beta e_{t,i} - a_i$ in $\mathcal{V}_t$ is upper bounded by
\begin{align} \label{eq:dang_row_size_weaker}
\Big|\big\{ j \in \mathcal{V}_t: a_i(j) \in \{\pm 1\} \big\}\Big| \leq \frac{10 k}{e^\lambda} + O(\log^2 n) \leq \frac{10 k}{\log^{3} n} + O(\log^2 n) \leq O\Big(\frac{k}{\log^3 n}\Big) ,
\end{align}
where the last inequality follows as $k \geq \log^5 n$. 

As $v_t$ is supported on $\mathcal{V}_t$ and 
each non-zero entry of $2 \beta e_{t,i} - a_i$ has absolute value at most $1.1$,
\begin{align*}
\E \langle 2 \beta e_{t,i} - a_i, v_t \rangle^2 
&= (2 \beta e_{t,i} - a_i)^\top \big(\E v_t v_t^\top \big) (2 \beta e_{t,i} - a_i)  \\
&\leq O\Big(\frac{1}{n_t} \Big) \cdot \|2 \beta e_{t,i} - a_i\|_2^2 =
O\Big(\frac{\|a_i\|_2^2}{n_t} \Big) 
\leq O\Big(\frac{k}{n_t \log^3 n}  \Big),
\end{align*}
where the first inequality uses that $\E[v_t v_t^\top] \preceq O\Big(\frac{1}{n_t} \Big) \cdot I$ by Lemma \ref{lem:cov_bound}.

 By \Cref{lem:slack_lb_unblocked}, for any row (dangerous or not) we have 
\begin{equation}
\label{eq:slack-lb-general}
    s_i(t) \geq \lambda b_0/ \log n.
\end{equation}
Therefore, 
\begin{align*}
\alpha_i(t) \cdot \E \langle 2 \beta e_{t,i} - a_i, v_t \rangle^2 
& \leq \frac{\lambda b_0}{s_i(t)^2} \cdot O\Big(\frac{k}{n_t \log^3 n}  \Big) 
 \leq O(1) \cdot \frac{\log^2 n}{\lambda b_0} \cdot \frac{k}{n_t \log^3 n} \leq \frac{c_t}{2} ,
\end{align*}
where the last inequality uses definition of $c_t$ in \eqref{eq:c_t_setting}. This proves the lemma.
\end{proof}
We remark that the upper bound in \eqref{eq:dang_row_size_weaker} is the only place in the entire proof of \Cref{thm:beck-fiala-main} where we use the assumption that $k \geq \log^5 n$.

\subsection{Total Potential Stays Bounded} 

Now we are ready to prove that our total potential $\Phi(t)$ stays bounded throughout our process.

\begin{lemma}[Total potential bound] \label{lem:total_potential_bound}
For the process defined in \Cref{sec:defn_process}, with high probability, for all time steps $t \in [0,n]$, we have $\Phi(t) \leq 10 \Phi(0)$.
\end{lemma}

\begin{proof}
The proof is similar to that of \Cref{lem:conc_col_weight}. We prove the stronger statement that with high probability, for all time steps $t \in [0,n]$, it holds that
\begin{enumerate}
    \item [(i)] $W_j(t) \leq k e^{2\lambda} + O( e^{3\lambda} \log^2 n)$ for all $j \in \mathcal{V}_t$, and 
    \item [(ii)] $\Phi(t) \leq 10 \Phi(0)$. 
\end{enumerate}
Let $\tau_{\fail} > 0$ be the first time that one of these two conditions fail, or define it to be $n$ if both conditions hold for all $t \in [0,n]$. We think of $\tau_{\fail}$ as a stopping time, and consider the modified process where $W_j(t)$ and $\Phi(t)$ remain fixed for all $t \geq \tau_{\fail}$.
Abusing notation slightly, we let $W_j(t)$ and $\Phi_i(t)$  denote these modified processes.

Consider any time step $t < \tau_{\fail}$. We may assume without loss of generality that $n_t \geq \log^6 n$, as otherwise our process freezes. The rows in $[n]$ at time $t$ can be divided into three parts:
\begin{enumerate}
    \item [(i)] For each large rows $i \in [n]$, their potential $\Phi_i(t)$ remains unchanged.  Denote the set of large rows as $\mathcal{L}_t \subset \mathcal{B}_t$. 
    \item [(ii)] For each small but blocked row $i \in \mathcal{B}_t \setminus \mathcal{L}_t$, we have $\langle 2 \beta e_{t,i} - a_i, v_t \rangle = 0$. So $\Phi_i(t)$ only has a negative drift and thus $\Phi(t)$ on decreases. 
    \item [(iii)] The change $d\Phi_i(t)$ for each row $i \notin \mathcal{B}_t$ contain both random martingale and drift terms.  
\end{enumerate}
We will control $\Phi(t)$ by comparing the terms $\E[d\Phi(t)]$ and $\E[(d\Phi(t))^2]$ and again using the inequality in \Cref{fact:freedman_conc}, as in the proof of \Cref{lem:conc_col_weight}.
Using the definitions above together with \eqref{eq:row_weight_change}, 
\begin{align*}
d \Phi(t) &\approx \sum_{i \notin \mathcal{B}_t} \gamma_i(t) \langle a_i - 2 \beta e_{t,i}, v_t \rangle \sqrt{dt}  \\
&\qquad  -\sum_{i \notin \mathcal{L}_t} \gamma_i(t) \cdot \Big( c_t + \beta \cdot \langle a_i^2 , v_t^2 \rangle - \alpha_i(t) \cdot \langle 2 \beta e_{t,i} - a_i, v_t \rangle^2 \Big) dt .
\end{align*}
Then the first moment of $d \Phi(t)$ can be bounded as
\begin{align} \label{eq:Phi_change_drift}
\E[d \Phi(t)] &= - \sum_{i \notin \mathcal{L}_t} \gamma_i(t) \cdot \Big( c_t + \beta \cdot \E \langle a_i^2 , v_t^2 \rangle - \alpha_i(t) \cdot \E \langle 2 \beta e_{t,i} - a_i, v_t \rangle^2 \Big) dt\nonumber  
\leq - \sum_{i \notin \mathcal{L}_t} \gamma_i(t) \cdot \frac{c_t}{2} dt ,
\end{align}
where the safe and dangerous rows in $[n] \setminus \mathcal{L}_t$ are bounded using \Cref{lem:safe_rows_energy,lem:dang_rows_ct} respectively. 
Using \Cref{lem:cov_bound}, the second moment of $d \Phi(t)$ can be similarly bounded as
\begin{align*}
\E[(d \Phi(t))^2] & = \E \Big( \sum_{i \notin \mathcal{B}_t} \gamma_i(t) \langle a_i - 2 \beta e_{t,i}, v_t \rangle\Big)^2 d t \nonumber 
\leq O \Big(\frac{k}{n_t}\Big) \cdot \sum_{i \notin \mathcal{B}_t} \gamma_i(t)^2 d t.
\end{align*}
For each unblocked row $i \notin \mathcal{B}_t$, using the lower bound on slack in \eqref{eq:slack-lb-general},
\begin{align*}
O\Big(\frac{k}{n_t}\Big) \cdot \gamma_i(t) 
&= O\Big(\frac{k}{n_t}\Big) \cdot \frac{\lambda b_0}{s_i(t)^2} \cdot \Phi_i(t) \\
&\leq O\Big(\frac{k}{n_t}\Big) \cdot \frac{\lambda \log^2 n}{\lambda^2 b_0} \cdot \frac{\Phi(0)}{n_t} \ll \frac{\Phi(0)}{\log^2 n} \cdot \frac{c_t}{2} .
\end{align*}
where the first inequality uses that the $n_t/10$ rows with the largest weights are blocked and that $\Phi(t) \leq 10 \Phi(0)$ as $t < \tau_{\fail}$, and the last inequality uses that $c_t = \Theta(\lambda k/(b_0n_t\log n))$ in \eqref{eq:c_t_setting} and our assumption that $n_t \geq \log^6 n$.
Consequently, we have 
\[
\E[d \Phi(t)] \leq - \Omega\Big( \frac{\log^2 n}{\Phi(0)} \Big) \cdot \E[(d \Phi(t))^2]  .
\]
Since our modified process freezes after $\tau_{\fail}$,
it then follows from \Cref{fact:freedman_conc} that for the modified process, with high probability, $\Phi(t) \leq 10 \Phi(0)$ holds for all time steps $t \in [0,n]$. This together with \Cref{lem:conc_col_weight} implies the result. 
\end{proof}

\subsection{Putting Everything Together: Proof of Main Result}

We can now put everything together and prove \Cref{thm:beck-fiala-main}. 

\begin{proof}[Proof of \Cref{thm:beck-fiala-main}]
We first modify the matrix $A$ so that the conditions in \Cref{assump:simplify_assump} hold. 

Then we start from $x_0 = 0$, and run the random walk process defined in \Cref{subsec:exp_barrier_potential,sec:defn_process}. By \Cref{lem:total_potential_bound}, with high probability we always have $\Phi(t) \leq 10 \Phi(0)$. Assume this bound holds for all time $t$, then \Cref{lem:disc_bound} implies that at time $n$, any row $i \in [n]$ satisfies 
\[
\langle a_i, x_n \rangle \leq b_n \leq O(\sqrt{k \log \log n}) .
\]
By \Cref{lem:round_full_coloring}, the fractional coloring $x_n$ can be further rounded to a full coloring $\widetilde{x} \in \{\pm 1\}^n$ with discrepancy $O(\sqrt{k \log \log n})$. This proves the theorem. 
\end{proof}

\section*{Acknowledgements}
We thank Marcus C.~Gozon for some useful discussions during the project. We also thank Thomas Rothvoss for useful comments on the paper. 

\bibliographystyle{alpha}
\bibliography{bib.bib}

\end{document}